\documentclass[11pt]{amsart}
\usepackage{graphicx}
\usepackage[active]{srcltx}
 \makeatletter
\renewcommand*\subjclass[2][2000]{%
  \def\@subjclass{#2}%
  \@ifundefined{subjclassname@#1}{%
    \ClassWarning{\@classname}{Unknown edition (#1) of Mathematics
      Subject Classification; using '1991'.}%
  }{%
    \@xp\let\@xp\subjclassname\csname subjclassname@#1\endcsname
  }%
}
 \makeatother
\usepackage{enumerate,url,amssymb,  mathrsfs,yhmath}

\newtheorem{theorem}{Theorem}[section]
\newtheorem{lemma}[theorem]{Lemma}
\newtheorem*{lemma*}{Lemma}

\theoremstyle{definition}

\theoremstyle{remark}
\newtheorem{remark}[theorem]{Remark}

\numberwithin{equation}{section}

\def\XXint#1#2#3{{\setbox0=\hbox{$#1{#2#3}{\int}$}
\vcenter{\hbox{$#2#3$}}\kern-.5\wd0}}

\def\le{\leqslant}
\def\ge{\geqslant}
\begin{document}

\title{Optimal estimates for
harmonic functions in the unit ball} \subjclass{Primary 31A05;
Secondary 42B30 }


\author{David Kalaj}
\address{University of Montenegro, Faculty of Natural Sciences and
Mathematics, Cetinjski put b.b. 81000 Podgorica, Montenegro}
\email{davidk@ac.me}

\author{Marijan Markovi\'c}
\address{University of Montenegro, Faculty of Natural Sciences and
Mathematics, Cetinjski put b.b. 81000 Podgorica, Montenegro}
\email{marijanmmarkovic@gmail.com}

\begin{abstract}
We find the sharp constants $C_p$ and the sharp functions
$C_p=C_p(x)$ in the inequality
$$|u(x)|\leq \frac{C_p}{(1-|x|^2)^{(n-1)/p}}\|u\|_{h^p(B^n)}, u\in h^p(B^n), x\in B^n,$$
in terms of Gauss hypergeometric and Euler functions. This extends
and improves some results of Axler, Bourdon and Ramey (\cite{ABR}),
where they obtained similar results which are sharp only in the
cases $p=2$ and $p=1$.
\end{abstract}  \maketitle


\section{Introduction and  statement of the results}

Let $n\ge 2$ and let $h^p(B^n),$ $1\le p\le \infty$, be the harmonic
Hardy spaces on the unit $n$-dimensional ball $B^n$ in the Euclidean
space $\mathbf{R}^n$, the space of all harmonic functions $u$
satisfying growth condition
$$\|u\|_p^p:=\|u\|^p_{h^p(B^n)}=\sup_{0<r<1}\int_S|u(r\zeta)|^pd\sigma(\zeta)<\infty$$
where $S=S^{n-1}$ is the unit sphere and $\sigma$ is the unique
normalized rotation invariant Borel measure on $S$. It is well known
that a harmonic function $u\in h^p(B^n)$ posses radial (angular)
limit $u(\zeta)$ in almost all points on sphere $\zeta\in S^{n-1}$
and that for $p>1$ it is possible to express in the form
\begin{equation}\label{poi}u(x)=\int_SP(x,\zeta)u(\zeta)d\sigma(\zeta),\end{equation}
where
$$P(x,\zeta)=\frac{1-|x|^2}{|x-\zeta|^n}, \zeta\in S $$
is Poisson kernel.

The maximum principle implies that, if $u\in h^\infty(B^n)$, then
$|u(x)|\le \|u\|_{\infty}$. On the other hand, it follows from the
Poisson representation formula \eqref{poi} that, if $u \in
h^1(B^n)$, then $$|u(x)|\le \sup_{\zeta \in S}P(x,\zeta) \|u\|_1.$$
Then $$\sup_{\zeta \in S}P(x,\zeta)=
\frac{(1+|x|)^n}{(1-|x|^2)^{n-1}}.$$ In this work we find a
representation for the sharp constants  $C_p$ and the sharp
functions $C_p=C_p(x)$ in the inequality $$|u(x)|\le
\frac{C_p}{(1-|x|^2)^{(n-1)/p}}\|u\|_p$$ where $x$ is an arbitrary
point in the unit ball $B^n$.

It is well known  that $C_p(x)$ is a bounded function in $B^n$ for
$1\le p\le \infty$, and the power $(n-1)/p$ is optimal. See
\cite[Proposition~6.16]{ABR} for the case $n\ge 2$ and $1\le p\le
\infty$ and \cite[Lemma~5.1.1]{lib} for the case of analytic
functions ($n=2$) and $0<p<\infty$. In the case when $h^p(B^n)$ is
Hilbert space, that is for $p=q=2$ we have next sharp point estimate
\begin{equation}\label{ola}|u(x)|\leq \sqrt \frac{1+|x|^2}{(1-|x|^2)^{n-1}}
\|u\|_{h^2(B^n)}.\end{equation}

The previous inequality is obtained in \cite[Proposition~6.23]{ABR}
using the fact that $h^2(B^n)$ is Hilbert space and Riesz
representation theorem for functionals by use of the scalar product
and Cauchy-Schwartz inequality in the general setting
$|\left<x,y\right>|\leq \|x\| \|y\|$.

Let $q$ be as usual conjugate with $p$ that is $1/p+1/q=1$. In this
paper we generalize \eqref{ola} by proving the following two
theorems.
\begin{theorem}\label{prima} Let $1<p\le \infty$. For all $u\in h^p(B^n)$
and $x\in B^n$ we have the following sharp inequality
\begin{equation}\label{ux}|u(x)|\leq
\frac{C_p(x)}{(1-|x|^2)^{(n-1)/p}}\|u\|_{h^p(B^n)}\end{equation}
where
$$C_p(x)=(F(-1 + n - \frac{n q}{2}, \frac 12 (n - n q),
  \frac n2, |x|^2))^{1/q},$$ and $F$ is the
Gauss Hypergeometric functions.
\end{theorem}
\begin{theorem}\label{seconda}  Let $1< p\le \infty$. For all $u\in h^p(B^n)$ and $x\in
B^n$ we have the sharp inequality
$$|u(x)|\leq \frac{C_p}{(1-|x|^2)^{(n-1)/p}}||u||_{h^p(B^n)}$$ where
$$C_p=\left\{
  \begin{array}{ll}
    1, & \hbox{if $q\le 2-\frac{2}{n}$;} \\
    \left(\frac{2^{nq-n} \Gamma(\frac n2) \Gamma\left(\frac{1+nq-n}{2}\right)}{\sqrt{\pi}
\Gamma(\frac{nq}{2})}\right)^{1/q}, & \hbox{if $q>2-\frac{2}{n}$.}
  \end{array}
\right.$$
\end{theorem}
\begin{remark} The cases $p=\infty$ and $p=1$ are already
considered in the introduction of this paper and are well-known. For
the case $p=\infty$ i.e. $q=1$ we have $C_p(x)=1$. For the case
$p=1$ we have $C_1(x)=(1+|x|)^n$. We will assume in the sequel that
$1<q<\infty$.  For $p=q=2$,
$$F(-1 + n - \frac{n q}{2}, \frac 12 (n - n q), \frac n2,
|x|^2)=1+|x|^2,$$ thus \eqref{ux} coincides with \eqref{ola}.
\end{remark}
\begin{remark}
If instead of $|u(x)|$ in Theorems~\ref{prima} and \ref{seconda}, we
put the norm of its gradient $|\nabla u(x)|$, then, instead of
$(n-1)/p$ we have the power $1+(n-1)/p$. See \cite{kha}, \cite{km},
\cite{km1}, \cite{kavu} and \cite{km2} for related results.
\end{remark}
\section{Representations for $C_p(x)$ as single integrals}
If $u$ is harmonic, and $T$ is an orthogonal transformation, then
$u\circ T$ is a harmonic function. Using this fact and $\|u\circ
T\|_{h^p(B^n)}=\|u\|_{h^p(B^n)}$ it is easy to see that
$C_p(x)=C_p(r_x)$ where $r_x=(|x|,0,\dots,0)$ is the vector on the
$e_1$ axis of the same norm as $x$.

In the sequel we will use M\"obius transform of the multidimensional
ball. Let us recall some basic facts from \cite{al}. In general a
M\"obius transform $T_x:B^n\to B^n$ has form
$$T_x(y)=\frac{(1-|x|^2)(y-x)-|y-x|^2x}{[y,x]^2}, \ \  y\in B^n$$
where $[y,x]=|y||x-y^*|$, $y^*=y/|y|^2$. In special case if $x=r_x$
we have
$$T_{r_x}(y)=(1-|x|^2)\frac{y-r_x}{|y-r_x|^2}-r_x.$$
Jacobi determinant of the $T_{r_x}:S\to S$ in the point $\eta \in S
$ is
$$J_{T_{r_x}}(\eta)=\left(\frac{1-|x|^2}{|\eta-r_x|^2}\right)^{n-1}.$$
By applying Holder inequality in the relation
$$u(x)=\int_SP(x,\zeta)u(\zeta)d\sigma(\zeta),$$ we have
$$|u(x)|\leq \left(\int_SP^q(x,\zeta)d\sigma(\zeta)\right)^{1/q} \|u\|_p.$$
Let $$I_q=\int_S P^q(x,\zeta)d\sigma(\zeta).$$ In the integral we
make change of variables $\zeta=-T_{r_x}(\eta),$ where
$$T_{r_x}(\eta)=(1-|x|^2)\frac{\eta-r_x}{|\eta-r_x|^2}-r_x$$ is M\"obius
transform of the unit ball $B^n$. Then
$$|r_x-\zeta|=\frac{1-|x|^2}{|\eta-r_x|}$$ and
$$d\sigma(\zeta)=\left(\frac{1-|x|^2}{|\eta-r_x|^2}\right)^{n-1}d\sigma(\eta).$$
So
\[\begin{split}
I_q&=\int_{S}\frac{(1-|x|^2)^q}{\frac{(1-|x|^2)^{nq}}{|r_x-\eta|^{nq}}}\left(\frac{1-|x|^2}{|\eta-r_x|^2}\right)^{n-1}d\sigma(\eta)
\\&=\int_{S}(1-|x|^2)^{q-nq+n-1}|\eta-r_x|^{nq-2n+2}d\sigma(\eta).
\end{split}\]
Further
\[\begin{split}
I_q^{1/q}=
\frac{1}{(1-|x|^2)^{(n-1)/p}}\left(\int_{S}|\eta-r_x|^{nq-2n+2}d\sigma(\eta)\right)^{1/q}
\end{split}\]
and
$$|u(x)|\leq \frac{1}{(1-|x|^2)^{(n-1)/p}}\left(\int_{S}|\eta-r_x|^{nq-2n+2}d\sigma(\eta)\right)^{1/q}\|u\|_{h^p(B^n)},$$
or
$$|u(x)|\leq \frac{C_p(x)}{(1-|x|^2)^{(n-1)/p}}\|u\|_{h^p(B^n)},$$
where
$$C_p(x)=\left(\int_{S}|\eta-r_x|^{nq-2n+2}d\sigma(\eta)\right)^{1/q}.$$
The sharp constant $C_p$ is
$$C_p=\sup_{x\in B^n}\left(\int_{S}|\eta-r_x|^{nq-2n+2}d\sigma(\eta)\right)^{1/q}.$$
For $n=2$ we have
\[\begin{split}
C_p^q(x)=\int_{S^1}|\eta-r_x|^{2q-2}d\sigma(\eta)=\frac{1}{2\pi}\int_0^{2\pi}(1+|x|^2-2|x|\cos\theta)^{q-1}d\theta
&\\=\frac{1}{\pi}\int_0^{\pi}(1+|x|^2-2|x|\cos\theta)^{q-1}d\theta.
\end{split}\]
Let $n>2$ and
$$K=\{(\theta_1,\dots,\theta_{n-2},\varphi):0\leq\theta_1,\dots,\theta_{n-2}\leq\pi,0\leq\varphi\leq 2\pi\}.$$
Using spherical coordinates
$(\eta_1,\dots,\eta_{n-1},\eta_{n})=(\theta_1,\dots,\theta_{n-2},\varphi)$
we get
\[\begin{split}C_p^q(x)&=\int_{S}|\eta-r_x|^{nq-2n+2}d\sigma(\eta)\\&=\frac{1}{\omega_{n-1}}\int_K(1+|x|^2-2|x|\cos\theta_1)^{nq/2-n+1}
\sin^{n-2}\theta_1\dots\sin\theta_{n-2}d\theta_1\dots d\varphi\\&
=\frac{2\pi}{\omega_{n-1}}I_n\int_0^{\pi}\sin^{n-2}\theta_1(1+|x|^2-2|x|\cos\theta_1)^{nq/2-n+1}d\theta_1,
\end{split}\]
where
$$I_n=\int_0^{\pi}\sin^{n-3}\theta_2d\theta_2\dots\int_0^{\pi}\sin\theta_{n-2}d\theta_{n-2}$$
and $\omega_{n-1}$ is volume of $n-1$-dimensional unit sphere. Since
$$\int_0^{\pi}\sin^{n-2}\theta
d\theta=\frac{\sqrt\pi\Gamma((n-1)/2)}{\Gamma(n/2)},$$ and $$I_n =
\frac{\omega_{n-1}}{2\pi \int_0^{\pi}\sin^{n-2}\theta d\theta}$$ we
have
\begin{equation}\label{cp}C_p^q(x)=\frac{\Gamma(n/2)}{\sqrt\pi\Gamma((n-1)/2)}\int_0^{\pi}\sin^{n-2}\theta(1+|x|^2-2|x|\cos\theta)^{nq/2-n+1}d\theta.\end{equation}
Note that (2.1) is also true for $n=2$ since $\Gamma(1/2)=\sqrt
\pi$.
\section{Representations for $C_p(x)$ as a Gauss hypergeometric function and the proof of Theorem~\ref{prima}}
We recall the classical definition of the Gauss hypergeometric
function: $${F}(a, b, c, z) = 1 +
\sum_{n=1}^\infty\frac{(a)_n(b)_n}{ (c)_n n!} z^n,$$ where
$(d)_n=d(d+1)\cdots (d+n-1)$ is the Pochhammer symbol. The series
converges at least for complex $z \in \mathbf{U}:=\{z:|z|<1\}\subset
\mathbf C$ and for $z\in \mathbf{T}:=\{z:|z|=1\}$, if $c>a+b$. Here
$\mathbf C$ is the complex plane. For $\Re (c)>\Re (b)>0$ we have
the following well-known formula
\begin{equation}\label{for}F(a,b,c,z)=\frac{\Gamma(c)}{\Gamma(b)\Gamma(c-b)}\int_0^1\frac{t^{b-1}(1-t)^{c-b-1}}{(1-tz)^a}dt.\end{equation}
It is easy to check the following formula
\begin{equation}\label{deri} \frac{d}{dz}{F}(a, b, c, z)=\frac{a
b \,{F}(1 + a, 1 + b, 1 + c, z)}{c}.\end{equation} We will use
Kummer's Quadratic Transformation of a hypergeometric
function,\begin{equation}\label{kum}
F(a,b,2b,\frac{4z}{(1+z)^2})=(1+z)^{2a}{F}(a,a+\frac
12-b,b+\frac{1}{2},z^2).\end{equation}
\begin{lemma}\label{pole}
For $n\ge 2$, $0\le r\le 1$ and $q\ge 1$ we have
\begin{equation}\label{mir}\begin{split}a_q(r):&=\int_0^\pi \sin^{n-2} t(1+r^2-2 r \cos
t)^{nq/2-n+1} dt \\&= \frac{\sqrt\pi
  \Gamma(\frac{ n-1}{2}) F(-1 + n - \frac{n q}{2}, \frac 12 (n - n q),
  \frac n2, r^2)}{ \Gamma(\frac n2)},\end{split}\end{equation} where $F$ is the
Gauss hypergeometric function.
  \end{lemma}
\begin{proof}[Proof of Lemma~\ref{pole}]
First of all $$ 1+r^2-2 r \cos t=(1+r)^2(1- z \frac{1+\cos t}{2}),$$
where $$ z=\frac{4r}{(1+r)^2}.$$ By taking the substitution $$ u
=\frac{1+\cos t}{2}$$ we obtain $$du = -\frac 12\sin t dt$$ and
$$\sin t = 2 u^{\frac 12}(1-u)^{\frac 12}$$ and therefore
\[\begin{split}a_q(r)&= 2(1+r)^{nq-2n+2}\int_0^1 \sin^{n-3} t (1- z
u)^{nq/2-n+1}du\\&= 2^{n-2}(1+r)^{nq-2n+2}\int_0^1
\frac{u^{\frac{n-3}{2}}(1-u)^{\frac{n-3}{2}}}{ (1- z
u)^{n-1-nq/2}}du.\end{split}\] By taking
$$a=n-1-nq/2,\ \ \  b=\frac{n-1}{2}\text{ and } c=2b =n-1$$  and by using the formula
\eqref{for}  and \eqref{kum}, we obtain
\[\begin{split}a_q(r)&=2^{n-2}(1+r)^{-2a}
\frac{\Gamma(b)\Gamma(c-b)}{\Gamma(c)}{F}(a,b,c,z)\\&=2^{n-2}(1+r)^{-2a}
\frac{\Gamma(b)\Gamma(c-b)}{\Gamma(c)}{F}(a,b,2b,\frac{4r}{(1+r)^2})\\&=2^{n-2}\frac{\Gamma(b)\Gamma(c-b)}{\Gamma(c)}{F}(a,a+\frac
12-b,b+\frac{1}{2},r^2).\end{split}\]  By using
$$2^{n-2}\frac{\Gamma(b)\Gamma(c-b)}{\Gamma(c)}=2^{n-2}\frac{\Gamma(\frac{n-1}{2})\Gamma(\frac{n-1}{2})}{\Gamma(n-1)}=\sqrt{\pi}
\frac{\Gamma(\frac{n - 1}2)}{\Gamma(\frac{n}{2})}$$ we obtain
finally \eqref{mir}.
\end{proof}
\begin{lemma}\label{lepa}
Under the conditions of Lemma~\ref{pole} we have $$a_q(s)=\left\{
                              \begin{array}{ll}
                               \frac{ \sqrt\pi
  \Gamma(\frac{n-1}{2})}{\Gamma(\frac n2)}, & \hbox{if $s=0$,} \\
                                \frac{2^{nq - n }
  \Gamma(\frac{n-1}{2}) \Gamma(\frac{1-n+nq}{2})}{ \Gamma(\frac{n q}{
  2})}, & \hbox{if $s=1$}.
                              \end{array}
                            \right.$$
\end{lemma}
\section{The proof of Theorem~\ref{seconda}} The proof of Theorem~\ref{seconda} lies on the following lemmas.
\begin{lemma}\label{cor} Let $1\le q\le 2$ and
$$a_q(r)=\int_0^\pi
\sin^{n-2} t(1+r^2-2 r \cos t)^{nq/2+n-1} dt. $$ Then
$$\max_{0\le r \le 1}a_q(r)=\left\{
                              \begin{array}{ll}
                                \frac{ \sqrt\pi
  \Gamma(\frac{n-1}{2})}{\Gamma(\frac n2)}, & \hbox{if $q<2-2/n$,} \\
                               \frac{2^{nq - n }
  \Gamma(\frac{n-1}{2}) \Gamma(\frac{1-n+nq}{2})}{ \Gamma(\frac{n q}{
  2})}, & \hbox{if $q\ge 2-2/n$}.
                              \end{array}
                          \right.$$
\end{lemma}
\begin{proof}[Proof of Lemma~\ref{cor}]
By using a well-known formula \eqref{deri} for the derivative of
Gauss hypergeometric function we obtain
\[\begin{split}g(r):&=\frac{d}{dr}\left(F(-1 + n - \frac{n q}{2}, \frac 12 (n
- n q),
  \frac n2, r^2)\right)\\&= r\cdot(2 + n (-2 + q)) (-1 + q) F(n - \frac{n q}{2},
  \frac 12 (2 + n - n q), \frac{2 + n}{2}, r^2).\end{split}\] Let
$$h(r)=F(n - \frac{n q}{2},
  \frac 12 (2 + n - n q), \frac{2 + n}{2}, r^2).$$ Put $b=n - \frac{n
q}{2}$. As $q<2$, then $b>0$. Further $$h(r)={F}( b, b+1 - n/2,
n/2+1, r^2).$$ Because $n/2+1>b> 0$, from \eqref{for} we obtain
\[\begin{split}h(r)&={F}( b+1 - n/2,b, n/2+1,
r^2)\\&=\frac{\Gamma(n/2+1)}{\Gamma(b)\Gamma(n/2+1-b)}\int_0^1\frac{t^{b-1}(1-t)^{n/2-b}}{(1-tr^2)^{
b+1 - n/2}}dt.\end{split}\] Hence \begin{equation}\label{pomo}
g(r)>0, \text{ for  all } 1< q <2 \text{ and } 0<r <1.
\end{equation}
On the other hand if $q<2-2/n$ then $$r\cdot(2 + n (-2 + q)) (-1 +
q)>0$$ and therefore $g(r)<0$. If $q>2-2/n$, then
$$r\cdot(2 + n (-2 + q)) (-1 + q)<0$$ and therefore $g(r)>0$. From \eqref{pomo} we obtain $$\max_{0\le
r \le 1}a_q(r)=\left\{
                              \begin{array}{ll}
                                a_q(0), & \hbox{if $q<2-2/n$;} \\
                                a_q(1), & \hbox{if $q\ge 2-2/n$}.
                              \end{array}
                            \right.$$ From Lemma~\ref{lepa} we
obtain the conclusion of the lemma.
\end{proof}
\begin{lemma}\label{lar} For $n\geq 2$ and $q\geq2$ integrals
$$\int_0^{\pi}\sin^{n-2}t(1+r^2-2r\cos t)^{nq/2-n+1}d t$$
are monotone increasing with respect to the parameter $r$, $0\leq
r\leq 1$.
\end{lemma}
\begin{proof} Let $$a_q(r)=\int_0^{\pi}\sin^{n-2}t(1+r^2-2r\cos
t)^{nq/2-n+1}d t.$$ For $0<r<1$ we have
\[\begin{split}a'(r)&=(nq-2n+2)\int_{0}^{\pi}\sin^{n-2} t(1+r^2-2r\cos
t)^{nq/2-n}(r-\cos t)d t\\& =(nq-2n+2)r\int_{0}^{\pi}\sin^{n-2}t
(1+r^2-2r\cos t)^{nq/2-q}d t\\& -(nq-2n+2)\int_{0}^{\pi}\sin^{n-2}
t(1+r^2-2r\cos t)^{nq/2-n}\cos t d t\geq 0 \end{split}\] because
\begin{equation}\label{bi}\int_{0}^{\pi}\sin^{n-2} t(1+r^2-2r\cos t)^{nq/2-n}d t \geq
0\end{equation} and
\begin{equation}\label{ib}\int_{0}^{\pi}\sin^{n-2}t \cos
t(1+r^2-2r\cos t)^{nq/2-n}d t\leq 0.\end{equation} The first
relation \eqref{bi} follows easily. The integral in \eqref{ib} can
be transformed:
\[\begin{split}&\int_{0}^{\pi}\sin^{n-2}t \cos t(1+r^2-2r\cos t)^{nq/2-n}d t
\\&=-\int_{-\pi/2}^{\pi/2}\cos^{n-2} t\sin t(1+r^2+2r\sin
t)^{nq/2-n}d t\\&=\int_{0}^{\pi/2}\cos^{n-2} t\sin t((1+r^2-2r\sin
t)^{nq/2-n}-(1+r^2+2r\sin t)^{nq/2-n})d t.\end{split}\] The
sub-integral expression is non-positive and consequently the
integral is also non-positive. Since $a_q(r)$ is monotone increasing
on the interval $(0,1)$ and continuous on the segment $[0,1]$ we
have conclusion.
\end{proof}
\begin{proof}[Proof of Theorem~\ref{seconda}]
By using Lemma~\ref{cor}, Lemma~\ref{lar} and \eqref{cp}, we have
$$C_p^q =1$$ if $q\leq2-2/n$ and
$$C_p^q= \frac{2^{nq-n}}{\sqrt\pi} \frac{\Gamma(\frac{n}{2}) \Gamma(\frac{1-n+nq}{2})}{ \Gamma(\frac{n q}{
  2})}$$  if $q> 2-2/n$.
\end{proof}
\begin{remark} Note that in the case $n=3$ we can
find very explicit sharp point estimate. Using classical
Newton-Leibnitz theorem we have for $r\neq0$
\[\begin{split}\int_0^{\pi}\sin t(1+r^2-2r\cos t)^{3q/2-2}d t&
=\frac{1}{2r(3q/2-1)}\int_0^{\pi}d_t(1+r^2-2r\cos t)^{3q/2-1}
\\&=\frac{1}{(3q-2)r}((1+r)^{3q-2}-(1-r)^{3q-2})\end{split}\]
So for $x\in B^3, x\neq0$, since $\frac{\Gamma(3/2)}{\sqrt\pi}=1/2$
$$|u(x)|\leq\frac{1}{(1-|x|^2)^{2/p}}\left(\frac{((1+|x|)^{3q-2}-(1-|x|)^{3q-2})}{2(3q-2)|x|}\right)^{1/q}\|u\|_{h^p(B^3)}.$$
For $x=0$ we have $|u(0)|\le\|u\|_{h^p(B^n)}$, $n\ge 2$.
\end{remark}


\begin{thebibliography}{99}
\bibitem{al} \textsc{L. Ahlfors:} {\it M\"obius transformations in several
dimensions.} Ordway Professorship Lectures in Mathematics.
University of Minnesota, School of Mathematics, Minneapolis, Minn.,
1981. ii+150 pp.
\bibitem{ABR}
\textsc{S. Axler, P. Bourdon and W. Ramey:} {\it Harmonic function
theory}, Springer Verlag New York 1992.
  \bibitem{cor}
\textsc{F. Colonna:} {\it The Bloch constant of bounded harmonic
mappings.} Indiana Univ. Math. J. \textbf{38} (1989), no. 4,
829--840.

\bibitem{kha}
\textsc{D. Khavinson:} {\it An extremal problem for harmonic
functions in the ball}, Canad. Math. Bull., \textbf{35} (1992),
218-220.

\bibitem{km}
\textsc{D. Kalaj, M. Markovic:} {\it Optimal estimates for the
gradient of harmonic functions in the unit disk} arXiv:1012.3153.

\bibitem{kavu}
\textsc{D. Kalaj, M. Vuorinen:} {\it On harmonic functions and the
Schwarz lemma,} to appear in Proceedings of the AMS.

\bibitem{km1}
\textsc{G. Kresin, V. Maz'ya:} {\it Sharp pointwise estimates for
directional derivatives of harmonic functions in a multidimensional
ball.} Journal of Mathematical Sciences, \textbf{169}, No. 2, 2010.

\bibitem{km2}
\textsc{A. J. Macintyre and W. W. Rogosinski:} {\it Extremum
problems in the theory of analytic functions}, Acta Math.
\textbf{82}, 1950, 275 - 325.

\bibitem{lib}
\textsc{M. Pavlovi\'c:} {\it Introduction to function spaces on the
disk.} 20. Matemati\v cki Institut SANU, Belgrade, 2004. vi+184 pp.
\end{thebibliography}
\end{document}